\documentclass[11 pt]{amsart}
\usepackage{stmaryrd}
\usepackage{graphicx}
\usepackage{amssymb}
\usepackage{amsmath}
\usepackage{amsthm}
\usepackage{amscd}
 \usepackage{amssymb}
\usepackage[all, knot]{xy}
\xyoption{arc}

\makeatother \makeatletter

\newtheorem{theorem}{Theorem}[section]
\newtheorem{lemma}[theorem]{Lemma}

\newtheorem{proposition}[theorem]{Proposition}
\newtheorem{remark}[theorem]{Remark}

\newtheorem{example}{Example}[section]
\newtheorem{problem}{Problem}[section]
 \input amssym.def
\input amssym.tex

\newcommand{\Rmnum}[1]{\expandafter\@slowromancap\romannumeral #1@}

\def\e{\operatorname{Ext}}

\begin{document}
\title{On the Structure of Tame Graded Basic Hopf Algebras II}
\author{Hua-Lin Huang}\thanks{The first author is supported by NSF (No. 10601052) and the second
author is supported by NSF (No. 10801069).  }
\address{School of Mathematics, Shandong University, Jinan, Shandong 250100, China}
\email{hualin@sdu.edu.cn}
\author{Gongxiang Liu}
\address{Department of Mathematics, Nanjing
University, Nanjing, Jiangsu 210093, China} \email{gxliu@nju.edu.cn}
\maketitle
\begin{abstract}
In continuation of the article \cite{L} we classify all radically
graded basic Hopf algebras of tame type over an algebraically closed
field of characteristic 0.
\end{abstract}

\section{Introduction}

Throughout this paper $k$ denotes an algebraically closed field and
all spaces are $k$-spaces. By an algebra we mean a finite
dimensional associative algebra with identity element. We freely use
the results, notations, and conventions of \cite{actions}.

According to the fundamental result of Drozd  \cite{Drozd}, every
finite dimensional algebra exactly belongs to one of following three
kinds of algebras: algebras of finite representation type, algebras
of tame type and wild algebras. For the algebras of the former two
kinds, a classification of indecomposable modules seems feasible. By
contrast, the module category of a wild algebra, being
``complicated" at least as that of any other algebra, can't afford
such a classification. Inspired by the Drozd's result, one is often
interested in classifying a given kind of algebras according to
their representation type. The class of finite-dimensional Hopf
algebras has been considered for quite a long time. For group
algebras of finite groups, the representation type of a block is
governed by its defect groups. A block of a finite-dimensional group
algebra is of finite representation type if and only if the
corresponding defect groups are cyclic while is tame if and only if
char$k=2$ and its defects groups are dihedral, semidihedral or
generalized quaternion. See
\cite{BDro}\cite{Benson}\cite{Karin}\cite{Hig}. In the case of small
quantum groups, i.e., Frobenius-Lusztig kernels, the only tame one
is $\textbf{u}_{q}(\mathfrak{sl}_{2})$ and the others are all wild
\cite{Su}\cite{Xiao}\cite{Cil}. The classification for
finite-dimensional cocommutative Hopf algebras, i.e., finite
algebraic groups, of finite representation type and tame type was
given by Farnsteiner and his cooperators recently
\cite{Far2}\cite{Far3}\cite{Far4}\cite{Far5}\cite{Far6}. The
representation theory of such cocommutative Hopf algebras was also
studied in \cite{Far1}\cite{Far7}.

Meanwhile,  basic Hopf algebras and their duals, pointed Hopf
algebras, have been studied intensively by many authors. See, for
example, \cite{AH}\cite{AH3}\cite{G2}. Our intention is to classify
finite dimensional basic Hopf algebras through their representation
type. In \cite{LL}, the authors have classified all basic Hopf
algebras of finite representation type and show that they are all
monomial Hopf algebras (see \cite{CHYZ}). For basic Hopf algebras of
tame type, the following result, up to the authors's knowledge, is
the best (see \cite{L}): Let $H$ be a basic Hopf algebra over an
algebraically closed field $k$ of characteristic different from 2,
then gr$H$ is tame if and only if gr$H\cong k<x,y>/I\times
(kG)^{\ast}$ for some finite group $G$ and some ideal $I$ which is
one of the following forms:

(1): $I=(x^{2}-y^{2},\;yx-ax^{2},\;xy)\;\;\;\;$ for $0\neq a\in k$;

(2): $I=(x^{2},\;y^{2},\;(xy)^{m}-a(yx)^{m})\;\;\;\;$ for $0\neq
a\in k$ and $m\geq 1$;

(3): $I=(x^{n}-y^{n},\;xy,\;yx)\;\;\;\;$ for $n\geq 2$;

(4): $I=(x^{2},\;y^{2},\;(xy)^{m}x-(yx)^{m}y)\;\;\;\;$ for $m\geq
1$;

(5): $I=(yx-x^{2},\;y^{2})$.

Here gr$H$ denotes the radically graded algebra of $H$ and
``$\times$" is the \emph{bosonization} defined in \cite{Majid} or
called \emph{biproduct} in \cite{Radford}. By this result, there are
at most five classes of tame graded basic Hopf algebras. By a
conclusion of Radford or Majid (see \cite{Radford}\cite{Majid}), for
an algebra $\Lambda$ and a finite group $G$, the bosonization
$\Lambda \times (kG)^{\ast}$ is a Hopf algebra if and only if
$\Lambda$ is a braided Hopf algebra in
$^{(kG)^{\ast}}_{(kG)^{\ast}}{\mathcal{YD}}$.  For an algebra
$k<x,y>/I$, the above conclusion dose not imply the existence of
finite group $G$ satisfying $k<x,y>/I$ is a braided Hopf algebra in
$^{(kG)^{\ast}}_{(kG)^{\ast}}{\mathcal{YD}}$. That's to say, for the
ideals $I$ listed in above conclusion, we don't know whether
$k<x,y>/I \times (kG)^{\ast}$ is a Hopf algebra or not! In fact,
this question is formulated as an open question posted in \cite{L}
(Problem 5.1):

\begin{problem} For a tame local graded Frobenius algebra $A$,
give an effective method to determine whether there is a finite
group $G$ such that $A$ is a braided Hopf algebra in
$^{(kG)^{\ast}}_{(kG)^{\ast}}{\mathcal{YD}}$. If such a $G$ exists,
then find all of them.
\end{problem}

In this paper, we will solve this problem. Indeed, we will show that
only the ideals $I=(x^{2},\;y^{2},\;(xy)^{m}-a(yx)^{m})$ can appear.
For more subtle description, see Theorems 4.9, 4.16, 5.2 and the
followed remarks. Then the class of tame graded basic Hopf algebras
can be classified completely.

The basic idea is simple. For a basic Hopf algebra $H$, we can
construct its radically graded version gr$H=H/J_{H}\oplus
J_{H}/J_{H}^{2}\oplus\cdots$. Then we establish the Gabriel's
theorem for graded basic Hopf algebras, that is, we show that there
is a Hopf surjection $kQ \rightarrow$ gr$ H$ where $Q$ is the
Gabriel quiver of gr$H$. By Theorem 2.3 of \cite{G2}, $Q$ is a
covering quiver $\Gamma_{G}(W)$. We find that $W$ consists of at
most two elements and the group generated by $W$ is automatically
abelian. Then we lift the ideals $I$ to the ideals $\widetilde{I}$
of the path algebra $k\Gamma_{G}(W)$ and the main difficulty is to
show when $\widetilde{I}$ is a Hopf ideal.

The paper is organized as follows. The next section contains all
knowledge that we need to go ahead. In particular, the works of
Green and Solberg on basic Hopf algebras are recalled and the
Gabriel's theorem for basic Hopf algebras is established. Some
combinatorial relations, which is the key to give a criterion to
determine an ideal to be a Hopf ideal, will be given in Section 3.
Section 4 deals with the classification of tame graded basic Hopf
algebras in the case of they are connected as algebras. Using
crossed products and the results gotten in Section 4, the class of
tame graded basic Hopf algebras are classified at the last section.

  \section{Preliminaries}
  In the following of this paper, we always assume that the
characteristic of $k$ is 0 unless otherwise stated.

  A \emph{quiver} is an oriented graph $Q=(Q_{0},\;Q_{1})$,
   where $Q_{0}$ denotes the set of vertices and $Q_{1}$
   denotes the set of arrows. $kQ$ denotes its \emph{path
   algebra}. An ideal $I$ of $kQ$
 is called \emph{admissible} if $J^{N}\subset I \subset J^{2}$ for
 some $N\geq 2$, where $J$ is the ideal generated by all arrows.

 For a basic algebra $A$, by the Gabriel's Theorem,
 there is a unique quiver $Q_{A}$,
 and an admissible ideal $I$ of $kQ_{A}$, such that $A\cong
 kQ_{A}/I$ (see \cite{ASS} and \cite{representation}). The quiver
 $Q_{A}$ is called the \emph{Gabriel quiver} or \emph{Ext-quiver} of $A$.

 Next, let us recall the definition of \emph{covering quivers} (see
\cite{G2}). Let $G$ be a finite group and let
$W=(w_{1},w_{2},\ldots,w_{n})$ be a sequence of elements of $G$. We
say $W$ is a \emph{weight sequence} if, for each $g\in G$, the
sequences $W$ and $(gw_{1}g^{-1},gw_{2}g^{-1},\ldots,gw_{n}g^{-1})$
are the same up to a permutation. In particular, $W$ is closed under
conjugation. Define a quiver, denoted by $\Gamma_{G}(W)$, as
follows. The vertices of $\Gamma_{G}(W)$ is the set $\{v_{g}\}_{g\in
G}$ and the arrows are given by
$$\{(a_{i},g):\;v_{g^{-1}}\rightarrow v_{w_{i}g^{-1}}\;|\; i=1,2,\ldots,n, g\in G\}$$
 We call this quiver the covering quiver (with
respect to $W$).

\begin{example}
\emph{Let $G=<g\;|\;g^{n}=1>$ and $W=(g,g)$, then the corresponding
covering quiver is}
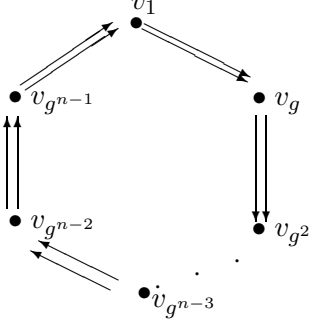
\begin{figure}[hbt]
\begin{picture}(350,100)(-120,0)
\put(42,100){\makebox(0,0){$\bullet$}}

\put(45,107){\makebox(0,0){$v_{1}$}}

\put(45,100){\vector(2,-1){40}}\put(44,98){\vector(2,-1){40}}
\put(95,70){\makebox(0,0){$\bullet\; v_{g}$}}

\put(91,65){\vector(0,-1){40}}\put(87,65){\vector(0,-1){40}}
\put(97,20){\makebox(0,0){$\bullet\; v_{g^{2}}$}}

\put(80,10){\makebox(0,0){$\cdot$}}
\put(65,5){\makebox(0,0){$\cdot$}}
\put(50,0){\makebox(0,0){$\cdot$}}
\put(45,-2){\makebox(0,0){$\bullet$}}
\put(60,-7){\makebox(0,0){$v_{g^{n-3}}$}}

\put(35,3){\vector(-2,1){30}}\put(32,-1){\vector(-2,1){30}}
\put(10,23){\makebox(0,0){$\bullet\;v_{g^{n-2}}$}}

\put(-3,30){\vector(0,1){35}}\put(-7,30){\vector(0,1){35}}
\put(10,70){\makebox(0,0){$\bullet\;v_{g^{n-1}}$}}

\put(0,75){\vector(3,2){35}}\put(-2,77){\vector(3,2){35}}
\end{picture}
\end{figure}

\emph{We denote this quiver by $\mathbb{Z}_{n}(2)$.}
\end{example}

\begin{lemma} Let $\Gamma_{G}(W)$ be a covering quiver. If the
length of  $W$ is 2, then the subgroup of $G$ generated by $W$ is an
abelian group.\end{lemma}
\begin{proof} Let $W=\{g,h\}$. Since $W$ is stable under the
conjugation, $ghg^{-1}=g$ or $ghg^{-1}=h$. If $ghg^{-1}=g$, then
$g=h$. If $ghg^{-1}=h$, then $gh=hg$. The lemma is clear
now.\end{proof}

 The following conclusion (see Theorem 2.3 in
\cite{G2}) states the importance of covering quivers.

\begin{lemma} Let $H$ be a finite dimensional basic Hopf algebra
over $k$. Then there exists a finite group $G$ and a weight sequence
$W=(w_{1},w_{2},\ldots,w_{n})$ of $G$, such that $H\cong
k\Gamma_{G}(W)/I$ for an admissible ideal $I$.
\end{lemma}

Let $\Gamma_{G}(W)$ be a covering quiver, a natural question is when
there is a Hopf structure on the path algebra $k\Gamma_{G}(W)$. To
answer this question, we need the concept \emph{allowable
$kG$-bimodule} which was introduced by Green and Solberg \cite{G2}.
Denote $V^{d}_{f}$ the $k$-space with basis the arrows from $v_{d}$
to $v_{f}$ for $d,f\in G$. We say a $kG$-bimodule structure on
$k\Gamma_{G}(W)$ is allowable if for any $g,d,f\in G$, the following
conditions hold:

(i) $g\cdot v_{f}=v_{fg^{-1}}$ and $v_{f}\cdot g=v_{g^{-1}f}$;

(ii)$g\cdot V^{d}_{f}\subset V^{dg^{-1}}_{fg^{-1}}$ and $
V^{d}_{f}\cdot g\subset V^{g^{-1}d}_{g^{-1}f}$.

For any vertex $v_{h}$ of $\Gamma_{G}(W)$ and any $x\in V^{d}_{f}$,
define three maps as follows:
$$\varepsilon(v_{h})=\left \{
\begin{array}{ll} 1, & \;\;\textmd{if}\; h=e\\
0 & \;\;\textrm{otherwise}
\end{array}\right.,\;\;\;\;\varepsilon(x)=0;$$
$$\Delta(v_{h})=\sum_{g\in G}v_{hg^{-1}}\otimes v_{g},\;\;\;\;
\Delta(x)=\sum_{g\in G}(g\cdot x\otimes v_{g}+ v_{g}\otimes x\cdot
g);$$
$$S(v_{h})=v_{h^{-1}},\;\;\;\;S(x)=-f\cdot x \cdot d.$$
The following conclusion is due to Green and Solberg (see Theorem
3.3 of \cite{G2}).
\begin{lemma} Suppose that $k\Gamma_{G}(W)$ has an allowable
$kG$-bimodule structure, then $k\Gamma_{G}(W)$ is a Hopf algebra
with counit $\varepsilon$, comultiplication $\Delta$ and antipode
$S$ given above.\end{lemma}

Let $H$ be a basic Hopf algebra, then its Jacobson radical $J_{H}$
is a Hopf ideal (see Lemma 1.1 in \cite{G2}). Hence $H/J_{H}\cong
(kG)^{\ast}$ for some finite group $G$ with counit $\varepsilon'$,
comultiplication $\Delta'$ and antipode $S'$ given in terms of the
dual basis $\{p_{g}\}_{g\in G}$ in $(kG)^{\ast}$ in the following
way:
$$\varepsilon'(p_{h})=\left \{
\begin{array}{ll} 1, & \;\;\textmd{if}\; h=e\\
0 & \;\;\textrm{otherwise}
\end{array}\right.;$$
$$\Delta'(p_{h})=\sum_{g\in G}p_{hg^{-1}}\otimes p_{g};$$
$$S'(v_{h})=v_{h^{-1}}.$$
The set $\{p_{g}\}_{g\in G}$ of primitive orthogonal idempotents in
$H/J_{H}$ can be lifted to a set of primitive orthogonal idempotents
$\{v_{g}\}_{g\in G}$ in $H$. Since $H^{\ast}$ can act on $H$
naturally (see \cite{actions}), $kG$ can act on $H$ now. Using this
it follows from the action of $kG$ on $H$ that $g\cdot
v_{f}=v_{fg^{-1}}$ and $v_{f}\cdot g=v_{g^{-1}f}$ modulo the
radical. Combining Lemma 1.2 and Lemma 2.1 in \cite{G2}, we have the
following result.

\begin{lemma} Let $H$ be a basic Hopf algebra. With the notations
above, the following assertions hold.

(a) The counit $\varepsilon$ for $H$ is given by
$$\varepsilon(v_{h})=\left \{
\begin{array}{ll} 1, & \;\;\textmd{if}\; h=e\\
0 & \;\;\textrm{otherwise}
\end{array}\right.$$
for all $h\in G$ and $\varepsilon(J_{H})=0$;

(b) The comultiplication for $H$ is given by
$$\Delta(v_{h})=\sum_{g\in G}v_{hg^{-1}}\otimes v_{g}$$
modulo $J_{H\otimes H}$ and for $x\in v_{f}J_{H}/J_{H}^{2} v_{d}$,
$$\Delta(x)=\sum_{g\in G}(g\cdot x\otimes v_{g}+ v_{g}\otimes x\cdot
g)$$ modulo $J_{H\otimes H}^{2}$.
\end{lemma}

Denote gr$H=H/J_{H}\oplus J_{H}/J_{H}^{2}\oplus\cdots $ the
radically graded algebra of $H$. By Lemma 5.1 of \cite{L}, it is
also a Hopf algebra. Now we can give the Gabriel's theorem for basic
Hopf algebras, which is indeed dual to Theorem 4.5 of \cite{FP}.

\begin{lemma} Let $H$ be a basic Hopf algebra and $\Gamma_{G}(W)$
its Gabriel quiver. Then there is a Hopf algebra surjection
$$\pi:\;\;k\Gamma_{G}(W)\longrightarrow\; \emph{gr}H$$
with $Ker \pi$  an admissible Hopf ideal of
$k\Gamma_{G}(W)$.\end{lemma}
\begin{proof} We use the notations above. At first, we must equip
$k\Gamma_{G}(W)$ with a Hopf structure. By Lemma 2.3, it is enough
to give an allowable $kG$-bimodule structure on $k\Gamma_{G}(W)$.
Indeed, for any vertex $v_{f}$, define $g\cdot v_{f}=v_{fg^{-1}}$
and $v_{f}\cdot g=v_{g^{-1}f}$. Transporting the left and right
actions of $kG$ on $v_{f}J_{H}/J_{H}^{2} v_{d}$ to the $k$-space
with the basis of all arrows from $v_{d}$ to $v_{f}$, we get the
left and right actions of $kG$ on paths of length 1. For a path
$p=\alpha_{n}\cdots \alpha_{1}$ of length $n$, define
$$g\cdot p=(g\cdot \alpha_{n})\cdots (g\cdot \alpha_{1}),\;\;\;\;
p\cdot g=(\alpha_{n}\cdot g)\cdots (\alpha_{1}\cdot g).$$ Thus we
get an allowable $kG$-bimodule structure on $k\Gamma_{G}(W)$ now and
the Hopf structure on $k\Gamma_{G}(W)$ is given through the way as
in the Lemma 2.3.

By the Gabriel's theorem, $\pi$ is an algebra surjection. We only
need to show that it is also a coalgebra map, i.e.
$\Delta\pi=(\pi\otimes \pi)\Delta$. Set $\phi_{1}=\Delta\pi$ and
$\phi_{2}=(\pi\otimes \pi)\Delta$. By Lemma 2.4, we have
$$\phi_{1}|_{k\Gamma_{G}(W)_{0}}=\phi_{2}|_{k\Gamma_{G}(W)_{0}},\;\;\;\;\phi_{1}|_{k\Gamma_{G}(W)_{1}}
=\phi_{2}|_{k\Gamma_{G}(W)_{1}}$$ where $k\Gamma_{G}(W)_{0}$ and
$k\Gamma_{G}(W)_{1}$ denote the k-spaces spanned by all vertices and
all arrows respectively. It is well-known that the path algebra is
indeed a tensor algebra. Using the universal property of tensor
algebra, we know that every algebra morphism $f$ from the path
algebra $k\Gamma$ is  determined uniquely by $f|_{k\Gamma_{0}}$ and
$f|_{k\Gamma_{1}}$. Thus $\phi_{1}=\phi_{2}$.
\end{proof}

Notice the difference between  Lemmas 2.2 and 2.5. Lemma 2.5 tells
us that the algebra isomorphism given in Lemma 2.2 can be
strengthened to be a Hopf isomorphism when the basic Hopf algebra is
radically graded. We also need Proposition 4.4 of \cite{G2}.

\begin{lemma} Let $k\Gamma_{G}(W)$ be a Hopf algebra with Hopf
structure given by an allowable $kG$-bimodule structure. Let
$I\subset k\Gamma_{G}(W)$ be a Hopf ideal which is admissible. Then
$I$ is stable under left and right $G$-actions.\end{lemma}

 At the end of this section, we give the definition of
 representation type. An algebra $A$ is said to be of \emph{finite representation type} provided
there are finitely many non-isomorphic indecomposable $A$-modules.
$A$ is of \emph{tame  type} or $A$ is a \emph{tame} algebra if $A$
is not of finite representation type, whereas for any dimension
$d>0$, there are finite number of $A$-$k[T]$-bimodules $M_{i}$ which
are free of finite rank as right $k[T]$-modules such that all but a
finite number of indecomposable $A$-modules of dimension $d$ are
isomorphic to $M_{i}\otimes_{k[T]}k[T]/(T-\lambda)$ for $\lambda\in
k$. We say that $A$ is of \emph{wild type} or $A$ is a \emph{wild}
algebra  if there is a finitely generated $A$-$k<X,Y>$-bimodule $B$
which is free as a right $k<X,Y>$-module such that the functor
$B\otimes_{k<X,Y>}-\;\;$ from mod-$k<X,Y>$, the category of finitely
generated $k<X,Y>$-modules, to mod-$A$, the category of finitely
generated $A$-modules, preserves indecomposability and reflects
isomorphisms. See
\cite{Drozd}\cite{DS}\cite{Karin}\cite{Craw}\cite{SS} for
 more details and in particular \cite{GNRSV} for geometric characterization of
 the tameness of algebras.

\section{Some Combinatorial Relations }
For our purpose, we need to consider the following combinatorial
functions: \begin{equation*} H_{1}(m,l, t)=\sum_{0\leq m_{1}\leq
m_{2}\leq \ldots \leq m_{l}\leq
m-l}t^{\sum_{i=1}^{l}m_{i}},\end{equation*}

 \begin{equation*}
\;\;\;\;H_{2}(m,l, t)=\sum_{0\leq n_{1}+n_{2}+\cdots + n_{l}\leq
m-l}t^{\sum_{i=1}^{l}(l+1-i)n_{i}},\end{equation*}
\begin{eqnarray*}
\;\;\;\;\;\;\;\;\;\;\;\;H_{3}(m,l, t)&=&t^{m-l}\sum_{0\leq n_{1}+
n_{2}+ \cdots + n_{l-1}\leq
m-l}t^{\sum_{i=1}^{l-1}(l-i)n_{i}}\\
&+& \sum_{0\leq n_{1}+n_{2}+\cdots + n_{l}\leq
m-l-1}t^{\sum_{i=1}^{l}(l+1-i)n_{i}}. \end{eqnarray*}
 Here $m,l\in
\mathbb{Z}^{+},\;0< l< m,\;m_{1},\ldots,m_{l}, n_{1},\ldots,n_{l}\in
\mathbb{N}$ and $t$ is an indeterminant.

\begin{lemma} $H_{1}(m,l, t)=H_{2}(m,l, t)=H_{3}(m,l, t)$\end{lemma}
\begin{proof} It is not hard to see that $H_{2}(m,l,t)=H_{3}(m,l,t)$. Now we show
that $H_{1}(m,l,t)=H_{2}(m,l,t)$. Note that
$$t^{\sum_{i=1}^{l}(l+1-i)n_{i}}=t^{n_{1}+(n_{1}+n_{2})+(n_{1}+n_{2}+n_{3})+
\cdots + (n_{1}+n_{2}+\cdots +n_{l})}.$$ Let
$m_{1}=n_{1},m_{2}=n_{1}+n_{2},\ldots, m_{l}=n_{1}+n_{2}+\cdots
+n_{l}$, we can see $H_{1}(m,l,t)=H_{2}(m,l,t)$.
\end{proof}

 Professor Zhi-Wei Sun gives us the proof of the main result of this
 section.

\begin{proposition}  $H_{1}(m,l,t)=0$ for all $0<l< m$ if and only if
$t$ is an $m$-th primitive root of unity.\end{proposition}
\begin{proof} $``\Longleftarrow"\;\;$ For any $0<l< m$, let $i_{j}=m_{j}+j$, then
$$H_{1}(m,l,t)=\sum_{1\leq i_{1} <i_{2}< \ldots <i_{l}\leq m}t^{\sum_{j=1}^{l}(i_{j}-j)}
=t^{-\frac{l(l+1)}{2}}\sum_{1\leq i_{1} <i_{2}< \ldots <i_{l}\leq m}
t^{\sum_{j=1}^{l}i_{j}}.$$ Consider the generating function
$$\prod_{r=1}^{m}(1+t^{r}x)=(1+tx)(1+t^{2}x)\cdots (1+t^{m}x),$$
where $x$ is an indeterminant. On one hand,
$$\prod_{r=1}^{m}(1+t^{r}x)=1+\sum_{l=1}^{m}(\sum_{1\leq i_{1} <i_{2}< \ldots <i_{l}\leq m}
t^{\sum_{j=1}^{l}i_{j}})x^{l}.$$ On the other hand,
$$\prod_{r=1}^{m}(1+t^{r}x)=\prod_{r=1}^{m}(1-t^{r}(-x)).$$
By using a well-known identity
$y^{m}-1=\prod_{r=0}^{m-1}(y-\zeta_{m}^{r})$ for any indeterminant
$y$ and $m$-th primitive root of unity $\zeta_{m}$, we see that
$$\prod_{r=1}^{m}(1+t^{r}x)=\prod_{r=1}^{m}(1-t^{r}(-x))=1-(-x)^{m}.$$
Thus for all $0<l< m$, $$\sum_{1\leq i_{1} <i_{2}< \ldots <i_{l}\leq
m} t^{\sum_{j=1}^{l}i_{j}}=0$$ and thus $H_{1}(m,l,t)=0$.

$``\Longrightarrow"\;\;$
 Let $l=1$, the condition implies
that
$$\sum_{i=0}^{m-1}t^{i}=0.$$
Thus $1\neq t$ is an $m$-th root of unity. There is no harm to
assume that $t$ is a $d$-th primitive root of unity with $d|m$.
Consider the generating function again,
$$\prod_{s=1}^{m}(1+t^{s}x)=\prod_{q=0}^{\frac{m}{d}-1}\prod_{r=1}^{d}(1-t^{qd+r}(-x)).$$
Just like the proof of sufficient part, we have
$$\prod_{q=0}^{\frac{m}{d}-1}\prod_{r=1}^{d}(1-t^{qd+r}(-x))=\prod_{q=0}^{\frac{m}{d}-1}(1-(-x)^{d})=
(1-(-x)^{d})^{\frac{m}{d}}.$$ By the proof of sufficiency, if $d<m$,
there must exist an $l$ with $0<l<m$ such that $$\sum_{1\leq i_{1}
<i_{2}< \ldots <i_{l}\leq m} t^{\sum_{j=1}^{l}i_{j}}\neq 0$$ and
thus $H_{1}(m,l,t)\neq 0$. It's a contradiction. So
$d=m$.\end{proof}

\section{Classification--Connected Case}
The main result of \cite{L} is the following result:
\begin{lemma}Let $H$ be a basic Hopf algebra, then\emph{ gr}$H$ is tame if and only if \emph{gr}$H\cong k<x,y>/I\times
(kG)^{\ast}$ for some finite group $G$ and some ideal $I$ which is
one of the following forms:

(1): $I=(x^{2}-y^{2},\;yx-ax^{2},\;xy)\;\;\;\;$ for $0\neq a\in k$;

(2): $I=(x^{2},\;y^{2},\;(xy)^{m}-a(yx)^{m})\;\;\;\;$ for $0\neq
a\in k$ and $m\geq 1$;

(3): $I=(x^{n}-y^{n},\;xy,\;yx)\;\;\;\;$ for $n\geq 2$;

(4): $I=(x^{2},\;y^{2},\;(xy)^{m}x-(yx)^{m}y)\;\;\;\;$ for $m\geq
1$;

(5): $I=(yx-x^{2},\;y^{2})$
\end{lemma}

As pointed out in the introduction, our aim is to determine which
ideals $I$ listed in Lemma 4.1 and what groups $G$ actually make
$k<x,y>/I \times (kG)^{\ast}$  a Hopf algebra.

At first, we show that the case (5) in Lemma 4.1 won't occur.

\begin{lemma}  $\Lambda=k<x,y>/(yx-x^{2},\;y^{2})$ is not a local Frobenius
algebra.
\end{lemma}
\begin{proof} Suppose it is.

Claim:  $J_{\Lambda}^{3}=(yxy) \subseteq soc \Lambda$. We have that
$J_{\Lambda}^{3}$ is generated by $xyx$ and $yxy$ by the given
relations. Moreover, modulo $J_{\Lambda}^{4}$ we have that
$xyx\equiv x^{3}\equiv yxx\equiv y^{2}x\equiv 0$ and therefore
$J_{\Lambda}^{3}=(yxy)$. This implies
$J_{\Lambda}^{4}=((yx)^{2})\subseteq yJ_{\Lambda}^{4}\subseteq
J_{\Lambda}^{5}$. Thus $J_{\Lambda}^{4}=0$ as required.

We claim $yxy$ must be zero now. Otherwise, assume $yxy\neq 0$ and
thus $J_{\Lambda}^{3}=(yxy) = soc \Lambda$. Since
$J_{\Lambda}^{4}=0$, we know $xyx=0$ and $xy^{2}=0$. This means
$xy\in soc\; \Lambda$. Clearly, $xy\neq 0$ since otherwise $yxy=0$.
Since $dim_{k}soc\; \Lambda=1$, there exists non-zero $c\in k$ such
that $xy=cyxy$. So we have $xy=cyxy=c^{2}y^{2}xy=0$. It's a
contradiction. This means $yxy=0$ and thus $J^{3}_{\Lambda}=0$ and
$J^{2}_{\Lambda}\subseteq soc \Lambda$. Therefor $soc \Lambda$ is
not simple, which is absurd.
\end{proof}
If there exists a finite group $G$ such that
$k<x,y>/(yx-x^{2},\;y^{2})\times (kG)^{\ast}$ is a Hopf algebra,
then $k<x,y>/(yx-x^{2},\;y^{2})$ must be local Frobenius
(Proposition 5.3 in \cite{L}). This implies that the case (5) can't
appear.

So we only need to consider cases (1)-(4). In this paper, we say a
basic Hopf algebra $H$ is graded if $H\cong$ gr$H$ as Hopf algebras.
Now let $H$ be a tame graded basic Hopf algebra and assume it is
connected as an algebra. In this situation, we say $H$ is a
connected tame graded basic Hopf algebra. Denote its Gabriel quiver
by $\Gamma_{G}(W)$, which is a covering quiver by Lemma 2.2. Thus
$H/J_{H}\cong (kG)^{\ast}$. By the assumption of $H$ being
connected, $\Gamma_{G}(W)$ is a connected quiver. From the
definition of covering quivers, we can deduce that $G=<W>$, the
group generated by $W$. By Lemma 4.1, the length of $W$ is 2. Thus
$G$ is an abelian group by Lemma 2.1. Combining these discussions,
we get the next observation.

\begin{proposition} Let $H$ be a connected tame graded basic Hopf
algebra and $\Gamma_{G}(W)$ its Gabriel quiver. Then the length of
$W$ is 2, $G=<W>$ and is abelian.\end{proposition}

By Lemma 2.5, the surjection $\pi: k\Gamma_{G}(W) \rightarrow H$ is
a Hopf algebra surjection and thus $Ker \pi$ is a Hopf ideal. We now
lift the ideals (1)-(4) in Lemma 4.1 to the ideals of
$k\Gamma_{G}(W)$ and we need to determine which lifting is a Hopf
ideal.

By Proposition 4.3, for any vertex of $\Gamma_{G}(W)$, there are
exactly two arrows going out and two arrows coming in. Denote the
arrows starting from $e$ by $a$ and $b$ respectively. Since $x,y$
are generators of the Jacobson radical of $k<x,y>/I$, we must lift
$x,y$ to linear combination of arrows. By Lemma 2.6, it is harmless
to lift $x$ and $y$ to $\sum_{g\in G}g\cdot a$ and $\sum_{g\in
G}g\cdot b$ respectively, i.e.,
$$x \mapsto \;X:=\sum_{g\in G}g\cdot a,\;\;\;\;y\mapsto\;Y:=\sum_{g\in G}g\cdot b.$$
Thus our task is just to determine whether the following ideals are
Hopf ideals or not:

(1): $I_{1}(a)=(X^{2}-Y^{2},\;YX-aX^{2},\;XY)\;\;\;\;$ for $0\neq
a\in k$;

(2): $I_{2}(m,a)=(X^{2},\;Y^{2},\;(XY)^{m}-a(YX)^{m})\;\;\;\;$ for
$0\neq a\in k$ and $m\geq 1$;

(3): $I_{3}(n)=(X^{n}-Y^{n},\;XY,\;YX)\;\;\;\;$ for $n\geq 2$;

(4): $I_{4}(m)=(X^{2},\;Y^{2},\;(XY)^{m}X-(YX)^{m}Y)\;\;\;\;$ for
$m\geq 1$.\\

By Proposition 4.3, $W=(g,g)$ or $W=(g,h)$ with $g\neq h$. We
discuss these two cases separately.

\subsection{Case 1: $W=(g,g)$}
Using the standard notations of covering quivers, $a=(a_{1}, e)$ and
$b=(a_{2},e)$. Assume that ord$(g)=n$ and $\Gamma_{G}(W)$ is just
the quiver given in Example 2.1. Since $G$ is abelian, the action of
$kG\otimes (kG)^{\textsf{op}}$ is diagonalizable. Thus, we can
assume that
$$g\cdot (a_{1}, e)=(a_{1}, g),\;\;\;\;g\cdot (a_{2}, e)=(a_{2}, g),$$
$$(a_{1}, e)\cdot g=q^{-1}g\cdot (a_{1}, e),\;\;\;\;(a_{2}, e)\cdot g=p^{-1}g\cdot (a_{2}, e)$$
for $p,q$ are $n$-th roots of unity. Denote $v_{g^{i}}$ by $v_{i}$
for simplicity.

\begin{lemma} $$\Delta(X)=X\otimes 1+(\sum_{i=0}^{n-1}q^{-i}v_{i})\otimes X,\;\;\;\;
S(X)=-q\sum_{i=0}^{n-1}a\cdot g^{i};$$
  $$\Delta(Y)=Y\otimes
1+(\sum_{i=0}^{n-1}p^{-i}v_{i})\otimes Y,\;\;\;\;
S(Y)=-p\sum_{i=0}^{n-1}b\cdot g^{i}.$$\end{lemma}
\begin{proof} We only prove the formulaes for $X$. Those for $Y$ can be proved
in the same manner.
\begin{eqnarray*}
\Delta(X)&=&\sum_{i=0}^{n-1}g^{i}\cdot X\otimes
v_{i}+\sum_{i=0}^{n-1}g^{i}\cdot v_{i}\otimes X\cdot g^{i}\\
&=&\sum_{i=0}^{n-1}X\otimes v_{i}+\sum_{i=0}^{n-1}v_{i}\otimes
(\sum_{j=0}^{n-1}g^{j}\cdot a\cdot g^{i})\\
&=& X\otimes 1+\sum_{i=0}^{n-1}v_{i}\otimes
(\sum_{j=0}^{n-1}q^{-i}g^{i+j}\cdot a)=X\otimes
1+\sum_{i=0}^{n-1}v_{i}\otimes q^{-i}X\\
&=&X\otimes 1+(\sum_{i=0}^{n-1}q^{-i}v_{i})\otimes X.
\end{eqnarray*}
And,
$$S(X)=S(\sum_{i=0}^{n-1}g^{-i}\cdot a)=\sum_{i=0}^{n-1}S(a)\cdot g^{i}=
\sum_{i=0}^{n-1}-(g\cdot a)\cdot g^{i}=-q\sum_{i=0}^{n-1}a\cdot
g^{i}.$$
 \end{proof}

For an indeterminant $x$, define the function $e_{x}:=
\sum_{i=0}^{n-1}x^{-i}v_{i}$.

\begin{lemma} We have the following identities
 $$Xe_{q}=qe_{q}X,\;\;\;\;Ye_{q}=qe_{q}Y,\;\;\;\;Xe_{p}=pe_{p}X,\;\;\;\;Ye_{p}=pe_{p}Y.$$
 \end{lemma}
 \begin{proof} Note that
 $$Xe_{q}=X\sum_{i=0}^{n-1}q^{-i}v_{i}=\sum_{i=0}^{n-1}q^{-i}g^{-i}\cdot a$$
 and $$qe_{q}X=q(\sum_{i=0}^{n-1}q^{-i}v_{i})X=q\sum_{i=0}^{n-1}q^{-i}g^{-(i-1)}\cdot a
 =\sum_{i=0}^{n-1}q^{-(i-1)}g^{-(i-1)}\cdot a.$$ Thus
 $Xe_{q}=qe_{q}X$. We can prove the other identities
 similarly.\end{proof}

 With the preparation, now we are ready to  determine whether
 $I_{1}(a),\;\\ I_{2}(m,a),\;I_{3}(n)$ and $I_{4}(m)$ are  Hopf ideals.

 \begin{lemma} $I_{1}(a)$ and $I_{3}(n)$ are not Hopf ideals of
 $k\Gamma_{G}(W)$.\end{lemma}
 \begin{proof} By Lemmas 4.4 and 4.5,
\begin{eqnarray*}\Delta(XY)&=&(X\otimes 1+e_{q}\otimes X)(Y\otimes 1+e_{p}\otimes
Y)\\
&=&XY\otimes 1
 +pe_{p}X\otimes Y+ e_{q}Y\otimes X+e_{pq}\otimes XY. \end{eqnarray*}
 Suppose $I_{1}(a)$ or $I_{3}(n)$ is a Hopf ideal, then clearly we have
 $$pe_{p}X\otimes Y+ e_{q}Y\otimes X=0,$$
 which is impossible.
\end{proof}

By Lemma 4.4 and Lemma 4.5, for any element $f(X,Y)$ generated by
$X,Y$, we can always write uniquely $\Delta(f(X,Y))$ in the
following form:
$$f(X,Y)\otimes 1+(f(X,Y))_{X}\otimes X+(f(X,Y))_{Y}\otimes Y+(f(X,Y))_{XY}\otimes XY+\cdots$$
$$+ (f(X,Y))_{(XY)^{i}}\otimes (XY)^{i}+(f(X,Y))_{(YX)^{i}}\otimes
(YX)^{i}+(f(X,Y))_{(XY)^{i}X}\otimes (XY)^{i}X$$
$$+(f(X,Y))_{Y(XY)^{i}}\otimes Y(XY)^{i}+\cdots.$$ In the following of this paper, we frequently use this
expression without any explanation.

\begin{lemma} $I_{4}(m)$ is not a Hopf ideal of $k\Gamma_{G}(W)$.
\end{lemma}
\begin{proof} Assume that it is a Hopf ideal.
It is not hard to see that
\begin{eqnarray*}&&((XY)^{m}X-(YX)^{m}Y))_{X}\otimes X\\
&\equiv& ((XY)^{m}e_{q}+e_{q}(YX)^{m})\otimes
X\;\;\;\;\textrm{mod}\;\;I_{4}(m)\otimes X.\end{eqnarray*}
Thus we
have $(XY)^{m}e_{q}+e_{q}(YX)^{m}\in I_{4}(m)$. This is absurd.
\end{proof}

\begin{lemma} $I_{2}(m,a)$ is a Hopf ideal if and only if
$m=1$ and $q^{-1}=a=p=-1$.\end{lemma}
\begin{proof}``$\Longrightarrow$"$\;\;$ Direct computations
show that
$$\Delta(X^{2})=X^{2}\otimes 1+(1+q)e_{q}X\otimes X+e_{q^{2}}\otimes X^{2}$$
and $$\Delta(Y^{2})=Y^{2}\otimes 1+(1+p)e_{p}Y\otimes
Y+e_{p^{2}}\otimes Y^{2}.$$ Thus $1+q=0=1+p$ and so
$$p=q=-1.$$
Next, we show that $m=1$. Otherwise, assume that $m>1$. In
$\Delta(XY)^{m}$, we have the following by direct computations,
\begin{eqnarray*}&&((XY)^{m})_{XY}\otimes XY\\
&\equiv&
(e_{q}(YX)^{m-1}e_{p}+\sum_{i=1}^{m}(XY)^{m-i}e_{pq}(XY)^{i-1})
\otimes XY\;\;\;\;\textrm{mod}\;\;I_{2}(m,a)\otimes
XY.\end{eqnarray*}
 By Lemma 4.5,
$\sum_{i=1}^{m}(XY)^{m-i}e_{pq}(XY)^{i-1}=\sum_{i=1}^{m}(p^{2}q^{2})^{m-i}e_{pq}(XY)^{m-1}$.
Similarly, in $\Delta(YX)^{m}$,
\begin{eqnarray*}&&((YX)^{m})_{XY}\otimes XY\\
&\equiv& \sum_{i=0}^{m-2}Y(XY)^{m-2-i}e_{pq}(XY)^{i}X\otimes XY
\;\;\;\;\textrm{mod}\;\;I_{2}(m,a)\otimes XY.\end{eqnarray*}
 Thus
$(e_{q}(YX)^{m-1}e_{p}+\sum_{i=1}^{m}(p^{2}q^{2})^{m-i}e_{pq}(XY)^{m-1})-a\sum_{i=0}^{m-2}Y(XY)^{m-2-i}\\e_{pq}(XY)^{i}X
\in I_{2}(m,a)$ which implies
$$\sum_{i=1}^{m}(p^{2}q^{2})^{m-i}=0.$$
This is impossible since $p=q=-1$. Thus $m=1$. Finally, we show that
$a=-1$. Indeed,
\begin{eqnarray*}\Delta(XY-aYX)&=&(XY-aYX)\otimes 1+(p-a)e_{p}X\otimes Y\\
&+&(1-aq)e_{q}Y\otimes X+e_{pq}\otimes (XY-aYX).\end{eqnarray*} So,
$p-a=0=1-aq$ which implies that $a=-1$.

``$\Longleftarrow$"$\;\;$ By the proof of necessity,
$$\Delta(I_{2}(1,-1))\subset I_{2}(1,-1)\otimes k\Gamma_{G}(W)+k\Gamma_{G}(W)\otimes I_{2}(1,-1).$$
We only need to show that $S(I_{2}(1,-1))\subset I_{2}(1,-1)$ and
$\varepsilon (I_{2}(1,-1))=0$. The verification of $\varepsilon
(I_{2}(1,-1))=0$ is trivial. And, by Lemma 4.4,
\begin{eqnarray*}
S(XY+YX)&=&S(Y)S(X)+S(X)S(Y)\\
&=&\sum_{i=0}^{n-1}(b\cdot g^{i})\sum_{i=0}^{n-1}(a\cdot
g^{i})+\sum_{i=0}^{n-1}(a\cdot
g^{i})\sum_{i=0}^{n-1}(b\cdot g^{i})\\
&=&\sum_{i=0}^{n-1}(b\cdot g^{i-1})(a\cdot g^{i})
+\sum_{i=0}^{n-1}(a\cdot g^{i-1})(b\cdot g^{i})\\
&=&-\sum_{i=0}^{n-1}v_{-i+2}(XY+YX)v_{-i}.
\end{eqnarray*}
That is, $S(I_{2}(1,-1))\subset I_{2}(1,-1)$.
\end{proof}

Recall  the quiver $\mathbb{Z}_{n}(2)$ given in Example 2.1.
Summarizing the previous arguments, we get the main result for case
1.
\begin{theorem}
Let $H$ be a connected tame graded basic Hopf algebra and
$\Gamma_{G}(W)$ its Gabriel quiver. If $W=(g,g)$, then as a Hopf
algebra,
$$H\cong k\mathbb{Z}_{n}(2)/(X^{2},Y^{2}, XY+YX)$$
for some even $n$. Here $X=\sum_{i=0}^{n-1}g^{i}\cdot (a_{1},e)$ and
$Y=\sum_{i=0}^{n-1}g^{i}\cdot (a_{2},e)$.\end{theorem}

\begin{remark} Note that $p,q$ are $n$-th roots of unity. By Lemma
4.8, $p=q=-1$ and thus $n$ must be an even. That's the reason why
$n$ is assumed to be even in the above theorem. Conversely, for any
cyclic group $G=<g|g^{n}=1>$ with $n$ an even, define the allowable
$kG$-bimodule on $k\mathbb{Z}_{n}(2)$ just as that given at the
beginning of this subsection. Then $k\mathbb{Z}_{n}(2)/(X^{2},Y^{2},
XY+YX)$ is a Hopf algebra by setting $p=q=-1$. Notice that this
indeed gives the answer to Problem 1.1 posted in Section 1 in this
case.
\end{remark}

\begin{example}
\emph{\textbf{(Book Algebras)}} \emph{Let $q$ be an $n$-th primitive
root of unity and $m$ a positive integer satisfying $(m,n)=1$. Let
$H=\textbf{h}(q,m)=k<y,x,g>/(x^{n},y^{n},g^{n}-1,gx-qxg,gy-q^{m}yg,xy-yx)$
 with comultiplication, antipode and counit given by}
$$\Delta(x)=x\otimes g+1\otimes x,\;\;\;\;\Delta(y)=y\otimes 1+g^{m}\otimes y,\;
\;\;\;\Delta(g)=g\otimes g$$
$$S(x)=-xg^{-1},\;\;S(y)=-g^{-m}y,\;\;S(g)=g^{-1},\;\;\varepsilon(x)=
\varepsilon(y)=0.\;\;\varepsilon(g)=1$$ \emph{It is a Hopf algebra
and called book algebra in \cite{AH}. It is a basic algebra since
$\textbf{h}(q,m)/J_{\textbf{h}(q,m)}$ is a commutative semisimple
algebra.} \emph{By Example 5.2 in \cite{L}, only $\textbf{h}(-1,1)$
is tame and the others are wild. }

\emph{Taking $n=2$ in Example 2.1, $\mathbb{Z}_{2}(2)$ is the
following quiver:}

\begin{figure}[hbt]
\begin{picture}(100,50)(0,0)
\put(0,25){\makebox(0,0){$ \bullet v_{e}$}}
\put(10,27){\vector(1,0){75}} \put(10,31){\vector(1,0){75}}
\put(85,23){\vector(-1,0){75}} \put(85,19){\vector(-1,0){75}}

\put(100,25){\makebox(0,0){$ \bullet v_{g}$}}
\end{picture}
\end{figure}
\emph{The allowable $k\mathbb{Z}_{2}$-bimodule structure on
$k\mathbb{Z}_{2}(2)$ is given by }
$$g\cdot v_{e}=v_{e}\cdot g=v_{g},\;\;\;
g\cdot v_{g}=v_{g}\cdot g=v_{e}$$
$$g\cdot (a_{1}, e)=(a_{1},
g)=-(a_{1}, e)\cdot g,\;\;\;g\cdot (a_{2}, e)=(a_{2}, g)=- (a_{2},
e)\cdot g.$$\emph{ Define}
  $\varphi:\;\;k\mathbb{Z}_{2}/(X^{2},Y^{2},XY+YX)\rightarrow \textbf{h}(-1,1)$
  by $$v_{e}\mapsto \frac{1}{2}(1+g),\;\;v_{g}\mapsto \frac{1}{2}(1-g),\;\;
  (a_{1},e)\mapsto xg\frac{1}{2}(1+g),\;\;(a_{2},e)\mapsto y\frac{1}{2}(1+g),$$
  $$(a_{1},g)\mapsto xg\frac{1}{2}(1-g),\;\;\;(a_{2},g)\mapsto y\frac{1}{2}(1-g).$$
\emph{It is straightforward to show that $\varphi$ is an isomorphism
of Hopf algebras, i.e.,}
$$k\mathbb{Z}_{2}/(X^{2},Y^{2},XY+YX)\cong \textbf{h}(-1,1).$$
\end{example}

\subsection{Case 2: $W=(g,h)$ with $g\neq h$}

Fix the covering quiver $\Gamma_{G}(W)$. Using the standard
notations of covering quivers, we can assume that $w_{1}=g$ and
$w_{2}=h$. Just like in the case 1,  we can assume that
$$g\cdot (a_{i}, e)=(a_{i}, g),\;\;\;\;h\cdot (a_{i}, e)=(a_{i}, h),$$
$$(a_{i}, e)\cdot g=q_{i}^{-1}g\cdot (a_{i}, e),\;\;\;\;(a_{i}, e)\cdot h=p_{i}^{-1}h\cdot (a_{i}, e)$$
for $i=1,2$ and $p_{i}^{\textrm{ord}(h)}=q_{i}^{\textrm{ord}(g)}=1$.
Abbreviate $v_{g^{i}h^{j}}$ as $v_{ij}$ for simplicity. For two
indeterminants $x,y$, define the function $e_{x,y}:=
\sum_{g^{i}h^{j}\in G}x^{-i}y^{-j}v_{ij}$. The proof of the
following is identical to that of Lemmas 4.4 and 4.5, so we state it
directly.

\begin{lemma}
$$\Delta(X)=X\otimes 1+ e_{q_{1},p_{1}}\otimes X,\;\;\;\;\Delta(Y)=Y\otimes 1+ e_{q_{2},p_{2}}\otimes Y.$$
$$Xe_{q_{1},p_{1}}=q_{1}e_{q_{1},p_{1}}X,\;\;
Xe_{q_{2},p_{2}}=q_{2}e_{q_{2},p_{2}}X.$$
$$Ye_{q_{1},p_{1}}=p_{1}e_{q_{1},p_{1}}Y,\;\;
Ye_{q_{2},p_{2}}=p_{2}e_{q_{2},p_{2}}Y.$$
\end{lemma}

It is also easy to see that Lemmas 4.6 and 4.7 are still true in
this case by using the same method.
\begin{lemma} $I_{1}(a), I_{3}(n)$ and $I_{4}(m)$ are not Hopf
ideals of $k\Gamma_{G}(W)$.
\end{lemma}

It remains to determine when $I_{2}(m,a)$ is a Hopf ideal.

\begin{lemma} If $I_{2}(m,a)$ is a Hopf ideal, then
$q_{1}=p_{2}=-1$ and
$a=(-1)^{m-1}q_{2}^{m}=(-1)^{m-1}p_{1}^{-m}$.\end{lemma}

\begin{proof} It follows by direct computations that
$$\Delta(X^{2})=X^{2}\otimes 1+(1+q_{1})e_{q_{1},p_{1}}X\otimes X+e_{q_{1}^{2},p_{1}^{2}}\otimes X^{2}$$
and $$\Delta(Y^{2})=Y^{2}\otimes 1+(1+p_{2})e_{q_{2},p_{2}}Y\otimes
Y+e_{q_{2}^{2},p_{2}^{2}}\otimes Y^{2}.$$ Thus $1+q_{1}=0=1+p_{2}$
and so $q_{1}=p_{2}=-1$.

Using the notation introduced before Lemma 4.7, we can see that
$$(XY)^{m}_{X}\otimes X\equiv e_{q_{1},p_{1}}(YX)^{m-1}Y\otimes X\;\;\textrm{mod}\;I_{2}(m,a)\otimes X$$
and
\begin{eqnarray*}(YX)^{m}_{X}\otimes X&\equiv& (YX)^{m-1}Ye_{q_{1},p_{1}}\otimes X\\
&=&p_{1}^{m} q_{1}^{m-1}e_{q_{1},p_{1}}(YX)^{m-1}Y\otimes
X\;\;\textrm{mod}\;I_{2}(m,a)\otimes X.\end{eqnarray*}
  Similarly,
\begin{eqnarray*}(XY)^{m}_{Y}\otimes Y&\equiv& (XY)^{m-1}Xe_{q_{2},p_{2}}\otimes Y\\
&=&q_{2}^{m}p_{2}^{m-1} e_{q_{2},p_{2}}(XY)^{m-1}X\otimes
Y\;\;\textrm{mod}\;I_{2}(m,a)\otimes Y\end{eqnarray*} and
$$(YX)^{m}_{Y}\otimes Y\equiv e_{q_{2},p_{2}}(XY)^{m-1}X\otimes Y\;\;\textrm{mod}\;I_{2}(m,a)\otimes X.$$
Thus $q_{2}^{m}p_{2}^{m-1}-a=1-ap_{1}^{m}q_{1}^{m-1}=0$ which
implies that $a=(-1)^{m-1}q_{2}^{m}=(-1)^{m-1}p_{1}^{-m}$.
\end{proof}

In the following, we need to use the functions defined at the
beginning of Section 3.

\begin{lemma} Let $0<l<m$, if $q_{1}=p_{2}=-1$ and
$a=(-1)^{m-1}q_{2}^{m}=(-1)^{m-1}p_{1}^{-m}$, then\\
\emph{(1)}
\begin{eqnarray*}&&(XY)^{m}_{(XY)^{l}}-a(YX)^{m}_{(XY)^{l}}\\
&\equiv& H_{2}(m,l,p_{1}q_{2})e_{(q_{1}q_{2})^{l},(p_{1}p_{2})^{l}}(XY)^{m-l}\\
&+&(-p_{1})^{l-m}H_{3}(m,l,p_{1}q_{2})e_{(q_{1}q_{2})^{l},(p_{1}p_{2})^{l}}(YX)^{m-l}
\;\;\;\;\;\emph{mod}\;I_{2}(m,a)
\end{eqnarray*}\\
\emph{(2)}
\begin{eqnarray*}
&&(XY)^{m}_{(YX)^{l}}-a(YX)^{m}_{(YX)^{l}}\\
&\equiv& -aH_{2}(m,l,p_{1}q_{2})e_{(q_{1}q_{2})^{l},(p_{1}p_{2})^{l}}(YX)^{m-l}\\
&-&a(-q_{2})^{l-m}H_{3}(m,l,p_{1}q_{2})e_{(q_{1}q_{2})^{l},(p_{1}p_{2})^{l}}(XY)^{m-l},
\;\;\;\;\;\emph{mod}\;I_{2}(m,a)
\end{eqnarray*}\\
\emph{ (3)}
\begin{eqnarray*}
&&(XY)^{m}_{Y(XY)^{l}}-a(YX)^{m}_{Y(XY)^{l}}\\
&\equiv&((-1)^{m-1}q_{2}^{m}-a)\sum_{0\leq n_{1}+n_{2}+\cdots +
n_{l}\leq
m-l-1}(p_{1}q_{2})^{\sum_{i=1}^{l}(l+1-i)n_{i}}\\
&&e_{q_{1}^{l}q_{2}^{l+1},p_{1}^{l}p_{2}^{l+1}}X(YX)^{m-l-1},\;\;\;\;\;\emph{mod}\;I_{2}(m,a)
\end{eqnarray*}\\
\emph{(4)}
\begin{eqnarray*}
&&(XY)^{m}_{(XY)^{l}X}-a(YX)^{m}_{(XY)^{l}X}\\
&\equiv&
(1-a(-1)^{m-1}p_{1}^{m})\sum_{0\leq n_{1}+n_{2}+\cdots + n_{l}\leq
m-l-1}(p_{1}q_{2})^{\sum_{i=1}^{l}(l+1-i)n_{i}}\\
&&e_{q_{1}^{l+1}q_{2}^{l},p_{1}^{l+1}p_{2}^{l}}(YX)^{m-l-1}Y,\;\;\;\;\;\emph{mod}\;I_{2}(m,a)
\end{eqnarray*}
\end{lemma}
\begin{proof} We only prove (1) because the others can be proved similarly.
 Since $X^{2}, Y^{2}\in
I_{2}(m,a)$, up to modulo $I_{2}(m,a)$, $X$ and $Y$ should appear
alternately in the left items in $(XY)^{m}_{(XY)^{l}}$. Thus there
are two possibilities: starting with $X$ or with $Y$. By this
observation, the items starting with $X$ are just
$$\sum_{0\leq n_{1}+n_{2}+\cdots n_{l}\leq m-l}
 (XY)^{n_{1}}e_{q_{1}q_{2},p_{1}p_{2}}(XY)^{n_{2}}e_{q_{1}q_{2},p_{1}p_{2}}\cdots
(XY)^{n_{l}}e_{q_{1}q_{2},p_{1}p_{2}}(XY)^{n_{l+1}}.$$ By iterated
application of Lemma 4.11, this item equals to
$$\sum_{0\leq n_{1}+n_{2}+\cdots n_{l}\leq m-l}
(p_{1}q_{2})^{n_{1}}(p_{1}q_{2})^{n_{1}+n_{2}}\cdots
(p_{1}q_{2})^{n_{1}+n_{2}+\cdots n_{l}}
e_{(q_{1}q_{2})^{l},(p_{1}p_{2})^{l}}(XY)^{m-l}$$ and thus equals to
$$H_{2}(m,l,p_{1}q_{2})e_{(q_{1}q_{2})^{l},(p_{1}p_{2})^{l}}(XY)^{m-l}.$$
Similarly, the items starting with $Y$ are just
$$\sum e_{q_{1},p_{1}}(YX)^{n_{1}}e_{q_{1}q_{2},p_{1}p_{2}}(YX)^{n_{2}}e_{q_{1}q_{2},p_{1}p_{2}}\cdots
(YX)^{n_{l-1}}e_{q_{1}q_{2},p_{1}p_{2}}(YX)^{n_{l}}e_{q_{2},p_{2}}$$
which equals to
\begin{eqnarray*}(q_{2}p_{2})^{m-l}\sum_{n_{1}+n_{2}+\cdots +n_{l-1}\leq
m-l}&& (p_{1}q_{2})^{n_{1}}(p_{1}q_{2})^{n_{1}+n_{2}}\cdots
(p_{1}q_{2})^{n_{1}+n_{2}+\cdots n_{l-1}}\\
&\cdot&
e_{(q_{1}q_{2})^{l},(p_{1}p_{2})^{l}}(XY)^{m-l}.\end{eqnarray*}
Meanwhile, all items in $(YX)^{m}_{(XY)^{l}}$ start from $Y$:
$$Y\sum (XY)^{n_{1}}e_{q_{1}q_{2},p_{1}p_{2}}(XY)^{n_{2}}e_{q_{1}q_{2},p_{1}p_{2}}\cdots
(XY)^{n_{l}}e_{q_{1}q_{2},p_{1}p_{2}}(XY)^{n_{l+1}}X$$ which equals
to
\begin{eqnarray*}(p_{1}p_{2})^{l}\sum_{0\leq n_{1}+n_{2}+\cdots n_{l}\leq
m-l-1}&& (p_{1}q_{2})^{n_{1}}(p_{1}q_{2})^{n_{1}+n_{2}}\cdots
(p_{1}q_{2})^{n_{1}+n_{2}+\cdots n_{l}}\\
&\cdot&
e_{(q_{1}q_{2})^{l},(p_{1}p_{2})^{l}}(YX)^{m-l}.\end{eqnarray*} Note
that $ q_{1}=p_{2}=-1$ and $a=(-1)^{m-1}p_{1}^{-m}$,
\begin{eqnarray*}
&&(q_{2}p_{2})^{m-l}\sum_{n_{1}+n_{2}+\cdots +n_{l-1}\leq m-l}
(p_{1}q_{2})^{n_{1}}(p_{1}q_{2})^{n_{1}+n_{2}}\cdots
(p_{1}q_{2})^{n_{1}+n_{2}+\cdots n_{l-1}}\\
&=&(-q_{2})^{m-l}\sum_{n_{1}+n_{2}+\cdots +n_{l-1}\leq
m-l}(p_{1}q_{2})^{\sum_{i=1}^{l-1}(l-i)n_{i}}\;\;\;\;\;\;(\star).
\end{eqnarray*}
And,
\begin{eqnarray*}&&-a(p_{1}p_{2})^{l}\sum_{0\leq n_{1}+n_{2}+\cdots n_{l}\leq m-l-1}
(p_{1}q_{2})^{n_{1}}(p_{1}q_{2})^{n_{1}+n_{2}}\cdots
(p_{1}q_{2})^{n_{1}+n_{2}+\cdots n_{l}}\\
&=&(-p_{1})^{l-m}\sum_{n_{1}+n_{2}+\cdots +n_{l}\leq
m-l-1}(p_{1}q_{2})^{\sum_{i=1}^{l}(l+1-i)n_{i}}\;\;\;\;\;\;(\ast).
\end{eqnarray*}
By the definition of $H_{3}(m,l,t)$, we see that
$$(\star)-(\ast)=(-p_{1})^{l-m}H_{3}(m,l,p_{1}q_{2})$$
and (1) is proved.
\end{proof}

\begin{proposition} $I_{2}(m,a)$ is a Hopf ideal if and only if

\emph{(1)} $q_{1}=p_{2}=-1$ and
$a=(-1)^{m-1}q_{2}^{m}=(-1)^{m-1}p_{1}^{-m}$;

\emph{(2)} $p_{1}q_{2}$ is an $m$-th primitive root of unity.
\end{proposition}

\begin{proof} ``$\Longrightarrow$"$\;\;$ (1) is just Lemma 4.13. By
Lemma 4.14, $H_{2}(m,l,p_{1}q_{2})=0$ for all $0< l<m$. Lemma 3.1
and Proposition 3.2 give us the desired conclusion.

``$\Longleftarrow$"$\;\;$ Using Lemma 3.1 and Proposition 3.2 again,
$H_{2}(m,l,p_{1}q_{2})=H_{3}(m,l,p_{1}q_{2})=0$. Then Lemmas 4.13
and 4.14 imply
$$\Delta(I_{2}(m,a))\subset I_{2}(m,a)\otimes k\Gamma_{G}(W)+k\Gamma_{G}(W)\otimes I_{2}(m,a).$$
By almost the same proof as in Lemma 4.8, we can show that
$$\varepsilon(I_{2}(m,a))=0,\;\;\;\;S(I_{2}(m,a))\subset
I_{2}(m,a).$$\end{proof}

\begin{theorem} Let $H$ be a connected tame graded basic Hopf
algebra and $\Gamma_{G}(W)$ its Gabriel quiver. If $W=(g,h)$ with
$g\neq h$, then under the assumption before the Lemma 4.11,
$$H\cong k\Gamma_{G}(W)/(X^{2},Y^{2},(XY)^{m}-(-1)^{m-1}q_{2}^{m}(YX)^{m})$$
as Hopf algebras for some $m> 0$.\end{theorem}

\begin{remark} \emph{(1)} Lemma 4.8 and Proposition 4.15 indeed give us the method to construct
all possible connected tame graded basic Hopf algebras and thus give
us some new examples of finite quantum groups.

\emph{(2)} If $I_{2}(m,a)$ is Hopf ideal of $k\Gamma_{G}(W)$ for
some $m$, then $m$ is factor of $l.c.m(\emph{ord}(g),
\emph{ord}(h))$, i.e., $m|\;l.c.m(\emph{ord}(g), \emph{ord}(h))$.
The reason is that $(p_{1}q_{2})^{l.c.m(\emph{ord}(g),
\emph{ord}(h))}=1$ and $p_{1}q_{2}$ is an $m$-th primitive root of
unity. Conversely, assume that $G$ is an abelian group generated by
$g,\;h\;$ ($g\neq h$) with $m|\;l.c.m(\emph{ord}(g),
\emph{ord}(h))$. Define the allowable $kG$-bimodule on
$k\Gamma_{G}((g,h))$ through the way given at the beginning of this
subsection. Let $q_{1}=p_{2}=-1$.  By a suitable choice of
$p_{1},\;q_{2}$, we can assume that $p_{1}q_{2}$ is an $m$-th
primitive root of unity. Now
$k\Gamma_{G}(W)/(X^{2},Y^{2},(XY)^{m}-(-1)^{m-1}q_{2}^{m}(YX)^{m})$
is an Hopf algebra. Notice that this is also give the answer to
Problem 1.1 in this case.
\end{remark}
At the end of this subsection, we recall a familiar example.

\begin{example} \emph{\textbf{(Tensor products of Taft algebras) }
 Let $T_{n^{2}}(q),\;T_{m^{2}}(q')$ be two Taft
algebras.  It is known that $T_{n^{2}}(q)\otimes_{k} T_{m^{2}}(q')$
is tame if and only if $m=n=2$ (see Example 5.1 in \cite{L}). Let
$G=\mathbb{Z}_{2}\times \mathbb{Z}_{2}\cong
<g,h\;|\;g^{2}=h^{2}=1,gh=hg>$ and the covering quiver
$\Gamma_{G}((g,h))$ is the following graph:}
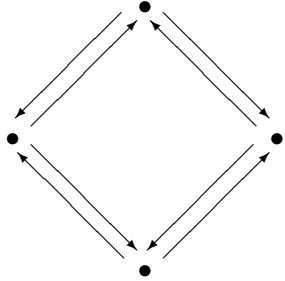
\begin{figure}[hbt]
\begin{picture}(100,100)(0,0)
\put(50,100){\makebox(0,0){$\bullet$}}
\put(50,0){\makebox(0,0){$\bullet$}}
\put(0,50){\makebox(0,0){$\bullet$}}\put(100,50){\makebox(0,0){$\bullet$}}

\put(57,98){\vector(1,-1){40}} \put(92,55){\vector(-1,1){40}}
\put(7,48){\vector(1,-1){40}} \put(42,5){\vector(-1,1){40}}

\put(41,98){\vector(-1,-1){40}} \put(7,55){\vector(1,1){40}}
\put(91,48){\vector(-1,-1){40}} \put(57,5){\vector(1,1){40}}

\end{picture}
\end{figure}

\emph{Through} $$(a_{1},e)\cdot g=-g\cdot
(a_{1},e),\;\;(a_{2},e)\cdot g=g\cdot (a_{2},e),$$
$$(a_{1},e)\cdot h=h\cdot (a_{1},e),\;\;(a_{2},e)\cdot h=-h\cdot
(a_{2},e),$$ $k\Gamma_{G}((g,h))$ \emph{is a Hopf algebra. Define
$\varphi:\;k\Gamma_{G}((g,h))/(X^{2},Y^{2},XY-YX)\rightarrow
T_{2^{2}}(-1)\otimes T_{2^{2}}(-1)$ by}
$$v_{e}\mapsto \frac{1}{2}(1+g)\frac{1}{2}(1+h),\;\;v_{g}\mapsto \frac{1}{2}(1-g)\frac{1}{2}(1+h),\;\;$$
$$v_{h}\mapsto \frac{1}{2}(1+g)\frac{1}{2}(1-h),\;\;v_{gh}\mapsto \frac{1}{2}(1-g)\frac{1}{2}(1-h),\;\;$$
$$(a_{1},e)\mapsto xg\frac{1}{2}(1+g)\frac{1}{2}(1+h),\;\;(a_{2},e)\mapsto yh\frac{1}{2}(1+g)\frac{1}{2}(1+h),\;\;$$
$$(a_{1},g)\mapsto xg\frac{1}{2}(1-g)\frac{1}{2}(1+h),\;\;(a_{2},g)\mapsto yh\frac{1}{2}(1-g)\frac{1}{2}(1+h),\;\;$$
$$(a_{1},h)\mapsto xg\frac{1}{2}(1+g)\frac{1}{2}(1-h),\;\;(a_{2},h)\mapsto yh\frac{1}{2}(1+g)\frac{1}{2}(1-h),\;\;$$
$$(a_{1},gh)\mapsto xg\frac{1}{2}(1-g)\frac{1}{2}(1-h),\;\;(a_{2},gh)\mapsto yh\frac{1}{2}(1-g)\frac{1}{2}(1-h).\;\;$$
\emph{It is straightforward to show that $\varphi$ induces an
isomorphism of Hopf algebras, i.e.,}
$$T_{2^{2}}(-1)\otimes T_{2^{2}}(-1)\cong k\Gamma_{G}((g,h))/(X^{2},Y^{2},XY-YX)$$
\emph{as Hopf algebras.}
\end{example}

\section{Classification-General Case}

Let $H$ be a graded basic Hopf algebra and $\Gamma_{G}(W)$  its
covering quiver. Let $N\subset G$ be the subgroup generated by $W$.
It is known that $k$ is an $H$-module through the counit map
$\varepsilon: \;H\rightarrow k$. We say a block of $H$ is the
\emph{principle block} if $k$, as a simple $H$-module, belongs to
this block. We denote this block by $H_{0}$.

\begin{proposition} \emph{(1)} $N$ is a normal subgroup of $G$;

\emph{(2)} As an algebra, $$H\cong H_{0}\oplus H_{0}\oplus \cdots
\oplus H_{0}$$ for $|G/N|$ copies of $H_{0}$;

\emph{(3)} $H_{0}$ is a Hopf algebra and is a Hopf quotient of $H$.
\end{proposition}
\begin{proof} (1) is obvious since $W$ is stable under the
conjugation.

Now we prove (2). By the Gabriel theorem for Hopf algebras (Lemma
2.5), there is a Hopf algebra isomorphism
$$k\Gamma_{G}(W)/I\cong H$$
with $I$ an admissible ideal of $k\Gamma_{G}(W)$. By the proof of
Lemma 2.5, the Hopf structure on $k\Gamma_{G}(W)$ is given by an
allowable $kG$-bimodule. Denote the connected component of
$\Gamma_{G}(W)$ containing $v_{e}$ by $\Gamma_{G}(W)^{\circ}$. By
the definition of covering quivers, every connected component of
$\Gamma_{G}(W)$ is indeed  $g\cdot \Gamma_{G}(W)^{\circ}$ for some
$g\in G$. It is easy to see that $g\cdot
\Gamma_{G}(W)^{\circ}=\Gamma_{G}(W)^{\circ}$ if and only if $g\in
N$. Thus there are $|G/N|$ numbers of connected components.

Let $I^{\circ}:=k\Gamma_{G}(W)^{\circ}\cap I$ and thus $H_{0}\cong
k\Gamma_{G}(W)^{\circ}/I^{\circ}$.  Using Lemma 2.6, $I$ is stable
under $G$-action. By the definition of allowable $kG$-bimodule,
$k(g\cdot \Gamma_{G}(W)^{\circ})\cap I$ is exactly $g\cdot
I^{\circ}$. This fact implies any block of $H$ must equal to $g\cdot
H_{0}$ and thus is isomorphic to $H_{0}$. (2) is proved.

At last, let's prove (3). For $h\in G$, it is known that
$\Delta(v_{h})=\sum_{g\in G}v_{g}\otimes v_{g^{-1}h}$. This implies
$\sum_{g\not\in N}kv_{g}$ generates a Hopf ideal of
$k\Gamma_{G}(W)$. Thus
$$H_{0}\cong k\Gamma_{G}(W)/(I,\Sigma_{g\not\in N}kv_{g})$$
is a Hopf algebra which clearly is a Hopf quotient of $H$.
\end{proof}

The structure of tame graded basic Hopf algebras can be determined
now. For a Hopf algebra $H$, let $H^{\ast}$ denote its dual.

\begin{theorem} Let $H$ be a tame graded basic Hopf algebra and $\Gamma_{G}(W)$  its
Gabriel quiver. Denote by $H_{0}$ the principle block of $H$ and
$\Gamma_{G}(W)^{\circ}$ the connected component of $\Gamma_{G}(W)$
containing $v_{e}$. Let $N\subset G$ be the subgroup generated by
$W$.

\emph{(1)} If $W=(g,g)$ for some $g\in G$, then as an algebra,
$$H\cong H_{0}\oplus H_{0}\oplus \cdots \oplus
H_{0}$$ for $|G/N|$ copies of $H_{0}$ and
$$H\cong (k(G/N))^{\ast}\#_{\sigma}(k \Gamma_{G}(W)^{\circ}/(X^{2},Y^{2},XY+YX))$$
as Hopf algebras where $X=\sum_{t\in N}t\cdot (a_{1},e)$ and
$Y=\sum_{t\in N}t\cdot (a_{2},e)$.

\emph{(2)} With the notations given in Subsection 4.2. If $W=(g,h)$
for some $g,h\in G$ and $g\neq h$, then as an algebra,
$$H\cong H_{0}\oplus H_{0}\oplus \cdots \oplus
H_{0}$$ for $|G/N|$ copies of $H_{0}$ and
$$H \cong (k(G/N))^{\ast}\#_{\sigma}(k\Gamma_{G}(W)^{\circ}/(X^{2},Y^{2},(XY)^{m}+(-q_{2})^{m}(YX)^{m}))$$
as Hopf algebras for some $m\in \mathbb{N}$ and $q_{2}\in k$ where
$X=\sum_{h\in N}h\cdot (a_{1},e)$ and $Y=\sum_{h\in N}h\cdot
(a_{2},e)$.
\end{theorem}
\begin{proof} Proposition 5.1 tells us that we have a Hopf
epimorphism $$\pi:\;\;H\rightarrow  H_{0}.$$ By a result of
Schneider \cite{Sch}, $$H\cong H^{co\pi}\#_{\sigma}H_{0}$$ where
$H^{co\pi}=\{a\in H\;|\;(id\otimes \pi)\Delta(a)=a\otimes 1\}$. It
is not hard to see that $H^{co\pi}=(k(G/N))^{\ast}$. Now the theorem
follows directly by Proposition 5.1, Theorem 4.9 and Theorem
4.16.\end{proof}

\begin{remark}
\emph{(1) }We can answer Problem 1.1 now. By this theorem, only some
special ideals of $\{(x^{2},y^{2}, (xy)^{m}-a(yx)^{m})\;| \;0\neq
a\in k, m\geq 1\}$ can appear and if one of them appears, then $G$
is necessary and sufficient to contain a normal subgroup $N$
satisfying the conditions given in Remark 4.10 or Remark 4.17 (2).

\emph{(2)} Almost all of computations of this paper are based on a
basic and simple observation, that is, the action of $kG\otimes
(kG)^{\textsf{op}}$ is diagonalizable (see paragraphs before Lemma
4.4 and Lemma 4.11) when $G$ is a finite abelian group. This is a
direct consequence of the assumption that $k$ is an algebraically
closed field of characteristic zero. Of course, if the
characteristic of $k$ is big enough to make $kG$ to be semisimple,
then our main results can also be established. Through developing a
suitable lifting method (see \cite{AH}\cite{AH3} for lifting of
pointed Hopf algebras), it is hopeful to get the classification of
all tame basic Hopf algebras over an algebraically closed field $k$
of characteristic zero at last. In general, the classification of
tame basic Hopf algebras (even radically graded) over an
algebraically closed field of \emph{positive} characteristic is
still an open and interesting question.

\emph{(3)} It is known that finite-dimensional Hopf algebras are
Frobenius algebras and of course they are selfinjective. The
classification of selfinjective algebras according to their
representation type over an algebraically closed field has been
researched for a long time. For the current stage of this subject,
see the survey article \cite{Sko}. The same question for tensor
product algebras, which are essential for Hopf algebras, has also
been investigated. In particular, all tame tensor product algebras
of nontrivial basic algebras over an algebraically closed field are
completely described \cite{Les}.

\end{remark}

\section*{Acknowledgements}

We are grateful to Professor Zhi-Wei Sun at Nanjing University for
providing the proof of Proposition 3.2 and to the referee for
his/her very helpful comments.


\begin{thebibliography}{20}
 \bibitem{AG} N. Andruskiewitsch, M. Gra$\widetilde{\textrm{n}}$a, Braided
 Hopf algebras over non-abelian groups, Bol. Acad. Ciencias (C$\acute{o}$rdoba)
 63 (1999), 45-78.
 \bibitem{AH1} N. Andruskiewitsch, H. -J. Schneider, Hopf algebras of
 Order $p^{2}$ and Braided Hopf Algebras of Order $p$, J. Algebra.
 199, 1998, 430-454.
 \bibitem{AH} N. Andruskiewitsch, H. -J. Schneider, Pointed Hopf
algebras, in ``New direction in Hopf algebras", 1-68, Math. Sci.
Res. Inst. Publ. 43, Cambridge Univ. Press, Cambridge, 2002.

\bibitem{AH3}  N. Andruskiewitsch, H. -J. Schneider, On the
classification of finite-dimensional pointed Hopf algebras, Ann.
Math., to appear.

\bibitem{ASS} I. Assem, D. Simson, A. Skowronski, Elements of the
representation theory of associative algebras 1: Techniques of
representation theory, London Mathematical Society Student Texts,
Vol 65, Cambridge University Press, 2006.

 \bibitem{representation} M. Auslander, I. Reiten,  Smal$\phi$,
Representation theory of artin algebras, Cambridge Studies in
Advanced Mathematics, Vol 36, Cambridge University Press, 1995.

\bibitem{Benson} D. Benson, Representation and cohomology I,
Cambridge Studies in Advanced Mathematics, Vol 30, Cambridge
University Press, 1991.

\bibitem{BDro} V. Bondarenko, Y. Drozd, Representation type of
fintie groups, J. Soviet Math. 20, 2515-2528, 1982.

\bibitem{Cil} C. Cibils, Half-quantum groups at roots of unity, Path
algebras, and reprsentation type, IMRN, 12(1997), 541-553.

\bibitem{CHYZ} Xiao-Wu Chen, Hua-Lin Huang, Yu Ye,  Pu Zhang,
Monomial Hopf algerbas, J. Algerba. 275(2004), 212-232.

\bibitem{Craw} W. W. Crawley-Boevey, Tame algebras and generic modules,
Proc. London Math. Soc. 63(1991), 241-264.

\bibitem{DS} P. Dowbor, A. Skowronski, On the representation type of locally bounded categories, Tsukuba J.
Math. 10(1986), no. 1, 63--72.

\bibitem{Drozd} Y. Drozd, Tame and wild matrix problems,
Representations and Quadratic Forms. Inst. Math., Acad. Sciences.
Ukrainian SSR, Kiev 1979, 39-74. Amer. Math. Soc. Transl. 128(1986),
31-55.

\bibitem{Karin} K. Erdmann, Blocks of tame representation type and
related algebras, Lecture Notes in Math. 1428, Springer-Verlag,
1990.

\bibitem{Far1} R. Farnsteiner, Auslander-Reiten components for Lie algebras of
reductive groups, Adv. Math 155(2000), 49-83.


\bibitem{Far7} R. Farnsteiner, On the Auslander-Reiten quiver of an
infinitesimal group, Nagoya Math. J. 160(2000), 103-121.

\bibitem{Far5} R. Farnsteiner, Polyhedral groups, Mckey quivers and the finite
algebraic groups with tame princial blocks, Invent. Math 166(2006),
27-94.

\bibitem{Far3} R. Farnsteiner, A.
Skowronski, Classification of restricted Lie algebras with tame
principal block, J. Reine. Angew. Math 546(2002), 1-45.

\bibitem{Far6} R. Farnsteiner, A. Skowronski, Galois actions and blocks of
tame infinitesimal group schmes, Trans. Amer. Math. Soc. 359, no.
12(2007), 5867-5898.

\bibitem{Far2} R. Farnsteiner,
D. Voigt, On cocommutative Hopf algebras of finite representation
type, Adv. Math 155(2000), 1-22.

\bibitem{Far4} R. Farnsteiner, D. Voigt, On infinitesimal groups of tame
representation type, Math. Z 244(2003), 479-513.

\bibitem{DSH} D. Fischman, S. Montgomery, H. -J. Schneider, Frobenius
extensions of subalgebras of Hopf algebras, Trans. Amer. Math. Soc.
349 (1997), 4857-4895.

\bibitem{GNRSV} P. Gabriel, L. A. Nazarova, A. V.  Roiter,
V. V. Sergeichuk, D. Vossieck,  Tame and wild subspace problems,
Ukrainian Math. J. 45(1993), no. 3, 313--352, Amer. Math. Soc.
Transl. in Ukrainian Math. J. 45(1993), no. 3, 335--372.

\bibitem{G2} E. Green, $\emptyset$. Solberg, Basic Hopf algebras
and quantum groups, Math. Z. 229(1998), 45-76.
\bibitem{Hig} D. Higman, Indecomposable representation at
characteristic p, Duke Math. J., 21, 377-381, 1954.

\bibitem{cosemisimple} R. G. Larson, D. E. Radford, Finite
dimensional cosemisimple Hopf algebras in characteristic 0 are
semisimple. J. Algebra (1988), 117, 267-289.

\bibitem{Les} Z. Leszczynski, A. Skowronski, Tame tensor products of
algebras, Colloq. Math. 98(2003), no. 1, 125-145.

\bibitem{L} G. X. Liu, On the structure of tame basic Hopf
algebras, J. Algebra. 299(2006), 841-853.
\bibitem{LL} G. X. Liu, F. Li, Pointed Hopf algebras of finite corpresentation type and their
classification, Proc. Amer. Math. Soc. 135(3)(2007), 649-657.
\bibitem{Majid} S. Majid, Crossed products by braided groups and bsonization,
J. Algebra 163 (1994), 165-190.
\bibitem{actions} S. Montgomery, Hopf algebras and their actions
on rings. \textbf{CBMS}, Lecture in Math.; Providence, RI, (1993);
Vol. 82.
\bibitem{MON} S. Montgomery, Indecomposable coalgebras, simple
comodules and pointed Hopf algebras, Proc. Amer. Math. Soc.
123(1995), 2343-2351.
\bibitem{FP} F. van Oystaeyen and P. Zhang, Quiver Hopf algebras,
J. Algebra. 280(2004), 577-589.
\bibitem{Radford} D. Radford, The structure of Hopf algebras with
a projection, J. Algebra 92 (1985), 322-347
\bibitem{Ringel} C. M. Ringel, The representation type of local
algebras, In Representation of Algebras, Lecture Notes in Math. 488,
Springer-Verlag, 1975, 282-305.

\bibitem{Sch} H. -J. Schneider, Normal basis and transitivity of crossed products for Hopf
algebras, J. Algebra 1152 (1992), 289-312.

\bibitem{SS} D. Simson, A. Skowro¨½ski,
 Elements of the representation theory of associative algebras 3:
 Representation-infinite tilted algebras, London Mathematical Society Student Texts, Vol 72,
  Cambridge University Press, 2007.

\bibitem{Sko} A. Skowronski, Selfinjective algebras: finite and tame
type,  Trends in representation theory of algebras and related
topics, Contemporary Math. 406, Amer. Math. Soc. (2006), 169-238.

\bibitem{Su} R. Suter, Modules over $U_{q}(\mathfrak{sl}_{2})$,
Comm. Math. Phy., 163(1994), 359-393.

\bibitem{Xiao} J. Xiao, Finite-dimensional representations of
$U_{t}(\mathfrak{sl}_{2})$ at roots of unity, Can. J.Math. vol.
49(4)(1997), 772-787.
\end{thebibliography}
\end{document}